\documentclass[11pt,reqno]{amsart}
\usepackage{amsfonts,amsmath,amssymb}
\newtheorem{thm}{Theorem}[section]

\newtheorem{corollary}[thm]{Corollary}
\newtheorem{lemma}[thm]{Lemma}
\newtheorem{definition}[thm]{Definition}
\newtheorem{example}[thm]{Example}
\newtheorem{remark}[thm]{Remark}

\usepackage{graphicx}
\usepackage{epstopdf}

\title [On  pairing on the Grothendieck ring of  Lie supergroup $GL(m,n)$.]{Canonical bilinear form  and Euler characters }

\author{A.N. Sergeev}\address{Department of Mathematics, Saratov State University, Astrakhanskaya 83, Saratov 410012   Russian Federation.}
 \email{SergeevAN@info.sgu.ru}

\begin{document}

\begin{abstract} An explicit formula for the canonical bilinear form on the Grothendieck ring of the Lie supergroup $GL(n,m)$ is given. As an application we  get an algorithm for the decomposition Euler  characters in terms of charctrers of irreducible modules in the category of partially polynomial modules. 
 \end{abstract}
\maketitle
\tableofcontents
\section{Introduction}
Let  $\mathcal F$ be the category of finite dimensional modules over  complex algebraic group $GL(n)$ and  $K(\mathcal F)$ be   its  Grothendieck ring.   There exists  a natural pairing on the ring  $K(\mathcal F)$
$$
([U],[V])= \dim Hom_{GL(n)}(U,V)
$$
Let us  identify  the ring  $K(\mathcal F)$ with the ring of symmetric Laurent polynomials 
$
\Bbb Z[x_1^{\pm1},\dots, x_n^{\pm1}]^{S_n}
$
 then the above pairing in  terms of characters  can be expressed in the following form
$$
([U],[V])=\left[(ch\, U)^* ch\,V \prod_{i\ne j}\left(1-\frac{x_i}{x_j}\right)\right]_0
$$
where $[f]_0$ means  the constant term of Laurent polynomial and $f^*(x_1,\dots,x_n)=f(x_1^{-1},\dots,x_n^{-1})$ (see \cite {Mac, Mac1}). This formula is very interesting from many points of view. 
On one side it allows to connect problems in representation theory to some problems with symmetric functions and their generalisations including Jack and Macdonald polynomials. On the other side  the above formula can be extended to the root system of any semisimple Lie algebra and the corresponding analogues of symmetric polynomials.
The main goal  of this paper is to prove the same kind of formula for complex algebraic Lie supergroup $GL(m,n)$ and to illustrate some of its applications. The category of finite dimensional representations of $GL(m,n)$ is not semisimple. Therefore   in this case  we have  the natural pairing only between projective modules $P(\mathcal F)$ and finite dimensional modules $K(\mathcal F)$ (see \cite{Brun, GS3})
$$
P(\mathcal F)\times K(\mathcal F)\longrightarrow \Bbb Z,\quad ([U],[V])= \dim Hom_{GL(m,n)}(U,V)
$$
where  $[L]$ means the class of module $L$ in appropriate version of Grothendieck ring. We also should mention that  the category of partially polynomials modules defined below is a convenient  object from the point of view of our bilinear form.
\section{Preliminaries}

Instead of  general linear complex algebraic supergroup $GL(m,n)$ it is more convenient to deal with the complex Lie superalgebra $\frak{gl}(m,n)$. So we will consider only finite dimensional representations of $\frak{gl}(n,m)$ such that every one of them  can be lifted to  the representation of $GL(m,n)$.

Let us remind that the Lie superalgebra $\frak{gl}(m,n)$ is the Lie  superalgebra of the  linear transformations of a $\Bbb Z_2$ graded vector space $V=V_0\oplus V_1$ ($V$ is also called the standard representation of $\frak g$). We have 
$$
\frak g_{\bar0}=\frak{gl}(m)\oplus\frak{gl}(n),\quad\frak g_{\bar1}=V_0\otimes V_1^*\oplus V_1\otimes V_0^*
$$ 
We also have $\Bbb Z$ graded decomposition 
$
\frak g=\frak g_{-1}\oplus \frak g_0\oplus \frak g_1
$
where
$$
\frak g_{-1}=V_1\otimes V^*_0,\,\, \frak g_{1}=V_0\otimes V^*_1,\,
$$
Let us fix  bases in $V_0=<e_1,\dots,e_m>$ and $V_1=<f_1,\dots,f_n>$ respectively. 
Let  $\frak{b}$ be the subalgebra of upper triangular matrix in $\frak{gl}(m,n)$ and $\frak{k}$ be  the    subalgebra of diagonal matrix in  $\frak{gl}(m,n)$ in the above basis. By   $\varepsilon_1,\dots,\varepsilon_m,\delta_1,\dots,\delta_n$  we will denote  the weights of standard representation with respect  to $\frak k$. The corresponding system of positive roots  $R^+=R^+_0\cup R^+_1$  of $\frak{gl}(m,n)$ can  be described in the following way
$$
R^+_{0}=\{\varepsilon_i-\varepsilon_j\,:\, 1\le i< j\le m\,:\, \delta_k-\delta_l,\, 1\le k<l\le n \}
$$
$$
R_1^+=\{\varepsilon_i-\delta_k,\, 1\le i\le m,\,1\le k\le n\}
$$

 Let also  
$$
P=\{\chi=\lambda_1\varepsilon_1+\dots+\lambda_m\varepsilon_m+\mu_1\delta_1+\dots+\mu_n\delta_n,\mid n_i,m_j\in \Bbb Z\}
$$
be the weight lattice and 
$$
P^+=\{\chi\in P\mid \lambda_i-\lambda_j\ge0,\,i<j\,:\mu_k-\mu_l\ge0,\,k<l\}
$$
be the set of highest weights.

We will use the following parity on the weight lattice due to C. Gruson  and V. Serganova \cite{GS1}  and Brundan and Stroppel \cite{BS} by saying  that $\varepsilon_i$ (resp. $\delta_j$) is even (resp. odd). It is easy to check that every finite dimensional module $L$ can be represented in the form
$$
L=L^+\oplus L^{-}
$$
where $L^+$ is the submodule of $L$ in which weight space has the same parity as the corresponding weight and $L^{-}$ is the submodule in which  the parities differ.  We should note that this construction is a particular case of Deligne construction category $Rep(G,z)$ from the paper \cite{D}  for $G=GL(m,n)$ and $z=diag(\underbrace{1,\dots,1}_{m},\underbrace{-1,\dots,-1}_{n}).$

Let us denote by $\mathcal F$ the category of finite dimensional modules over $\frak{gl}(n,m)$ such that every module in $\mathcal F$ is semisimple over Cartan  subalgebra $\frak k$ and and all  its weights are in $P$.
By  $K(\mathcal F )$ we will denote the quotient of the Grothendieck ring of   $\mathcal F$ by the relation $[L]-[\Pi(L)]=0$ where $\Pi(L)$ is the module  with the shifted parity $\Pi(L)_0=L_1, \Pi(L)_1=L_0$ and $x*v=(-1)^{p(x)}xv,x\in\mathfrak{gl}(m,n).$ For every $L\in \mathcal F$ we can define 
$$
ch\,L=\sum_{\chi}\dim L_{\chi}e^{\chi}
$$
where the sum is taken over all weights of $L$. It is easy to see that $ch\,L$ is well defined function on $K(\mathcal F)$.

The ring $K(\mathcal F)$ can be describe explicitly in the following way.  Let 
$$
P_{m,n}=\Bbb Z[x_1^{\pm1},\dots, x_m^{\pm1},\, y_1^{\pm1},\dots, y_n^{\pm1}]
$$
be the ring of Laurent polynomials  in variables $x_1,\dots,x_m$ and $y_1,\dots, y_n.$ 

If we set $x_i=e^{\varepsilon_i},\, y_j=e^{\delta_j}$ then we get a character map 
$$
ch : K(\mathcal F)\longrightarrow P_{m,n}
$$
 Let   also 
$$
\Lambda^{\pm}_{m,n}=\{f\in P_{m,n}^{S_m\times S_n}\mid x_i\frac{\partial f}{\partial x_i}+y_j\frac{\partial f}{\partial y_j}\in(x_i+y_j) \}
$$
be the subring of $P_{m.n}$ of supersymmetric Laurent  polynomials. 
\begin{thm}\cite{SV1} The   ring  $K(\mathcal F)$ is isomorphic to the ring $\Lambda^{\pm}_{m,n}$ under the character map.
\end{thm}
\begin{remark} Actually in the paper \cite{SV1}  slightly different versions of Grothen-dieck ring  and  the algebra  $\Lambda^{\pm}_{m,n}$  were considered. But it is easy to check that  they are isomorphic to our ones. We prefer to use characters  instead of supercharacters in this paper in order to avoid some unnecessary  signs.
\end{remark}
It will be needed later an explicit description of  the projective covers of the  irreducible finite dimensional modules due to Brundan \cite{Brun}. We give the description here in a slightly different way. 

First let us for  any $\chi\in P^+$ define a pair of sets 
$$
A=\{(\chi+\rho,\varepsilon_1),\dots, (\chi+\rho,\varepsilon_m)\},\quad B=\{(\chi+\rho,\delta_1),\dots, (\chi+\rho,\delta_n)\}
$$
where 
$$
\rho=\frac12\sum_{\alpha\in R_0^+}\alpha-\frac12\sum_{\alpha\in R_1^+}\alpha+\frac12(n-m+1)(\sum_{i=1}^m\varepsilon_i-\sum_{j=1}^n\delta_j)
$$
$$
=\sum_{i=1}^m(1-i)\varepsilon_i+\sum_{j=1}^n(m-j)\delta_j
$$
Our $\rho$ is slightly different from the standard one but it is more convenient since the elements of $A$ and $B$ are integers.
So instead of highest weights we will use the set of pairs $(A,B)$ such that  $A,B\subset\Bbb Z$ and $|A|=m,\,|B|=n$.  We will  also use the language of diagrams  which is due to Brundun and Stroppel \cite{BS} but we will use it  here  in a form due  to I. Musson and V. Serganova  \cite{MS}.

\begin{definition} Let $(A,B)$ be a pair of subsets in $\Bbb Z$ such that $|A|=m,\, |B|=n$. Then the  corresponding diagram is the following function on $\Bbb Z$
$$
f(x)=\begin{cases} \times,\,\,x\in A\cap B\\
\circ,\,\,x\in A'\cup B'\\
>,\,\,x\in A\setminus B\\
<,\,\, x\in B\setminus A
\end{cases}
$$
\end{definition}
Let us also set 
$$
\varphi(\times)=1,\,\varphi(\circ)=-1,\,\,\varphi(>)=\varphi(<)=0
$$
$$
[a,b]=\{c\in\Bbb Z\mid a\le c\le b\},\,[a,b)=\{c\in\Bbb Z\mid a\le c< b\},
$$
$$
\,(a,b)=\{c\in\Bbb Z\mid a< c<b\}
$$
and for integers $a<b$ let us define a transposition
$$
\pi_{a}^b:\Bbb Z\longrightarrow\Bbb Z,\quad \pi_a^b(x)=\begin{cases}x,\,x\ne a,b\\
b,\,x=a\\
a,\,x=b
\end{cases}
$$
\begin{definition}\label{defadm}  We will call a transposition $\pi_a^b$   an  admissible for $f$ if $a\in f^{-1}(\times),\,b\in f^{-1}(\circ)$ 
and  the following conditions are fulfilled
$$
 b>a,\,\, \sum_{i\in [a,b]}(\varphi\circ f)(i)=0,\,\, \sum_{i\in[a,c]}(\varphi\circ f)(i)>0,\, \text{for any}\,\, c\in[a,b).
$$
Since $b$ is uniquely defined by $f$ and $a$ we sometimes will omit $b$.
 \end{definition}
 
 The following Lemma easily follows from the  definition above.
 \begin{lemma}\label{prop} The following statements hold true
 
 $1)$ If $\pi_a^b,\, \pi_{c}^d$ are two admissible transpositions for $f$  then one of the following conditions is fulfilled
 \begin{equation}\label{cond0}
[a,b]\cap [c,d]=\emptyset,\,\,\, [a,b]\subset (c,d),\,\,\,[c,d]\subset (a,b)
\end{equation}
$2)$ Let $\pi_a^b$ be an admissible transposition for $f$ and $d\in[a,b]$ be such that $f(d)=\circ$. Then there exist an admissible  transposition for $f$  of the form $\pi_c^d$.

 \end{lemma}
 
\begin{corollary} Admissible transpositions pairwise commute.
\end{corollary}
\begin{proof} It  easily follows from Lemma \ref{prop}.
 \end{proof}
Now let us define for a diagram $f$ and  any  $C\subset f^{-1}(\times)$  the permutation of\, $\Bbb Z$\, by the formula 
\begin{equation}\label{prod}
\pi_C=\prod_{c\in C}\pi_c
\end{equation}
We should mention that the above product is well defined since admissible transpositions commute with each other. 
\begin{definition} Let $P(f)$ be the projective cover of irreducible module $L(f)$. We will denote  by $\mathcal P(f)$ the set of $g$ such that $K(g)$ is a subquotient  of $P(f)$.
\end{definition}
Now we can formulate the main result of Brundan \cite{Brun}. 
\begin{thm}\label{brundan}   $P(f)$ has a multiplicity free Kac flag and  
$$
\mathcal P(f)=\{g\mid g=\pi_C(f), C\subset f^{-1}(\times)\}
$$
\end{thm}
In order to  get an algorithm   for the decomposition  Kac modules into the sum of irreducible modules we need the following combinatorial Lemma.
\begin{lemma}\label{adm}  Let $f,g$ be such diagrams that
$$
g=\tau_r\circ\tau_2\circ\dots\circ\tau_1(f)
$$
where $\tau_i=\pi_{a_i}^{b_i} $ is  a transposition and $f_i=\tau_{i-1}\circ\dots\tau_1(f),\,i=1,\dots, r$. Suppose  also that for  any pair of $i>j$  we have
\begin{equation}\label{cond}
[a_i,b_i]\cap[a_j,b_j]=\emptyset, \, \text{or}\,\,\, [a_i,b_i]\subset (a_j,b_j)
\end{equation}
Then  for any $i=1,\dots r$ the transposition  $\tau_i$ is admissible for $f$ if and only if $\tau_i$ is admissible for $f_i,i=1,\dots,r$.
\end{lemma}

\begin{proof}   Let us prove first  that  functions $f_i$ and $f$ coincide on the segment $[a_i,b_i]$ for $1\le i\le r$. The following equalities are easy  to check
$$
\varphi\circ f_{i}=\varphi\circ f+2\sum_{j=1}^{i-1}(\delta_{b_j}-\delta_{a_j}),\,\,\, i=1,\dots, r
$$
$$
 f_i^{-1}(\times)=\left(f^{-1}(\times)\setminus\{a_1,\dots,a_{i-1}\}\right)\cup\{b_1,\dots,b_{i-1}\}
$$
If $t\in [a_i,b_i]$  then from the conditions of the Lemma  it follows that $\delta_{a_j}(t)=\delta_{b_j}(t)=0$ for any $1\le j<i$. Therefore $\varphi\circ f_i(t)=\varphi\circ f(t)$ on the segment $[a_i,b_i]$. Now Lemma follows from Definition \ref{defadm}.
\end{proof}

\begin{corollary}\label{adm1}  We  will keep the notations from Lemma \ref {adm}. Suppose that 
$
\tau_1=\pi_{a_1}^{b_1},\dots,\tau_r=\pi_{a_r}^{b_r}
$ 
is a set of  transpositions  such that $a_1<a_2<\dots<a_r$ and $a_i\ne b_j,\, 1\le i,j\le r$.
Suppose   also that  for any $i=1,\dots r$  transposition  $\tau_i$ is admissible for $f_i$. Then all transpositions $\tau_1,\dots,\tau_r$ are admissible for $f$.  
\end{corollary}
\begin{proof} Let us prove by induction that conditions of Lemma \ref{adm} are fulfilled. We will use induction on $r$. If $r=1$ then the statement of the Lemma is trivial. Let $r>1$. By inductive assumption 
transpositions $\tau_1,\dots,\tau_{r-1}$ are admissible for $f$. Therefore 
$$
f_{r}^{-1}(\times)=\left(f^{-1}(\times)\setminus\{a_1,\dots,a_{r-1}\}\right)\cup\{b_1,\dots,b_{r-1}\}
$$
Since $\tau_r$ is  admissible for $f_r$  we have $a_r\in f_{r}^{-1}(\times)$. By assumptions  of the Lemma  $a_r\ne b_1,\dots,b_{r-1}$ therefore $a_r\in f^{-1}(\times)$. Let $\pi_{a_r}^c$ be the corresponding admissible transposition  for $f$. Then for any $i\le r$  one of the following  conditions holds true
  $$
  [a_i,b_i]\cap [a_r,c]=\emptyset,\,\,[a_r,c]\subset (a_i,b_i),\,\, [a_i,b_i]\subset (a_r,c)
  $$
 The last condition is impossible since $a_r>a_i$. Therefore by Lemma \ref{adm} transposition $\pi_{a_r}^c$ is admissible for $f_r$. Therefore $\pi_{a_r}^{b_r}=\pi_{a_r}^c$ is admissible for $f$.
\end{proof}
 
\begin{corollary}\label{admis} Suppose that 
$
\tau_1=\pi_{a_1}^{b_1},\,\dots,\tau_r=\pi_{a_r}^{b_r}
$
is  a set of transpositions such that $ b_1>b_2>\dots>b_r$ and $a_i\ne b_j,\, 1\le i,j\le r$. Suppose also that      $\tau_i$ is admissible for  $f_i,\,i=1,\dots,n$. Then $\tau_i,\,i=1,\dots, r$ is admissible for $f$. 
\end{corollary}
\begin{proof} Let us prove by induction that conditions  (\ref{cond})  are fulfilled. We will use induction on $r$. If $r=1$ then the statement of the Lemma is trivial. Let $r>1$. By inductive assumption  the  conditions (\ref{cond}) are fulfilled therefore by Lemma \ref{adm}
transpositions $\tau_1,\dots,\tau_{r-1}$ are admissible for $f$. Therefore 
$$
f_{r}^{-1}(\times)=\left(f^{-1}(\times)\setminus\{a_1,\dots,a_{r-1}\}\right)\cup\{b_1,\dots,b_{r-1}\}
$$
Since $\tau_r$ is  admissible for $f_r$  we have $a_r\in f_{r}^{-1}(\times)$. By   our assumptions  $a_r\ne b_1,\dots,b_{r-1}$ therefore $a_r\in f^{-1}(\times)$. Besides since $b_r\ne a_1,\dots,a_{r-1}$ we have $b_r\in f^{-1}(\circ)$. Let $\pi_{a_r}^c$ be the corresponding admissible transposition  for $f$. Suppose that $b_r<c$, then $b_r\in [a_r,c]$. Therefore by Lemma \ref{prop} there exist  an admissible for $f$ transposition  $\pi_a^{b_r}$. Then for any $i< r$  one of the following  conditions holds true
  $$
  [a_i,b_i]\cap [a,b_r]=\emptyset,\,\,[a,b_r]\subset (a_i,b_i),\,\, [a_i,b_i]\subset (a, b_r)
  $$
 The last condition is impossible since $b_r<b_i$. Therefore by Lemma \ref{adm} transposition $\pi_{a}^{b_r}$ is admissible for $f_r$. Therefore $\pi_{a_r}^{b_r}=\pi_{a}^{b_r}$ is admissible for $f$. If $b_r\ge c$ then again condition $[a_i,b_i]\subset (a_r, c)$
 is impossible and $\pi_{a_r}^{b_r}=\pi_a^c$ is admissible for $f$.
\end{proof}
\begin{corollary}\label{irr}  Irreducible module  $L(f)$ is a subquotient  of Kac module  $K(g)$ if and only if there exist a sequence of  transpositions 
$$\sigma_1=\pi_{c_1}^{,d_1},\,\dots,\sigma_r=\pi_{c_r}^{d_r}
$$ where $ c_i<d_i,\,i=1,\dots,r$ such that

$1)$  $\sigma_i$ is admissible for $\sigma_i\circ\dots\circ\sigma_1(g)$, $i=1,\dots, r$ and $\sigma_r\circ\dots\circ\sigma_1(g)=f$

$2)$  $c_1>c_2>\dots>c_r$ 

$3)$ $c_i\ne d_j,\, 1\le i,j\le r $

\end{corollary}

\begin{proof} Suppose that all  conditions of the Corollary are fulfilled. Then
$$
g=\sigma_1\circ\sigma_{2}\circ\dots\circ\sigma_r(f)
$$
If we set $\tau_i=\sigma_{r+1-i},\,a_i=c_{r-i+1},\,b_i=d_{r-i+1}$ where $i=1,\dots,r$ then it is easy to see that all conditions of  Corollary \ref{adm1} are fulfilled. Therefore  $K(g)$ is a subquotient of $P(f)$. Therefore by $BGG$ reciprocity \cite{Z} $L(f)$ is a sub quotient of $K(g)$. 

Now let us suppose that $L(f)$ is a subquotient of Kac module $K(g)$. Then again by $BGG$ reciprocity $K(g)$ is a sub quotient of $P(f)$. Therefore by Theorem \ref{brundan}  
$
g=\pi_A(f),\, A\subset f^{-1}(\times).
$
Let $A=\{a_1,a_2,\dots, a_r\}$ where $a_1<a_2<\dots<a_r$.  Since admissible transpositions pairwise commute we have 
$$
g=\tau_r\circ\tau_2\circ\dots\circ\tau_1(f)
$$
where $\tau_i=\pi_{a_i}^{b_i}$. Let us check that  conditions (\ref{cond}) are fulfilled. It is enough to verify  that  inclusion $[a_j,b_j]\subset (a_i,b_i)$ is impossible if $i>j$. Indeed if it is so then  $a_j>a_i$ and we get a contradiction. Therefore by Lemma \ref{adm} $\tau_i$ is admissible for $f_i=\tau_{i-1}\circ\circ\dots\circ\tau_1(f)$ and we can set
$$
\sigma_{i}=\tau_{r-i+1},\,c_i=a_{r-i+1},\,d_i=b_{r-i+1}\,\,\,i=1,\dots,r.
$$
\end{proof}

In the same way we can prove the following Corollary.
\begin{corollary}\label{irr1}  Irreducible module  $L(f)$ is a subquotient  of Kac module  $K(g)$ if and only if there exist a sequence of  transpositions 
$$\sigma_1=\pi_{c_1}^{,d_1},\,\dots,\sigma_r=\pi_{c_r}^{d_r}
$$ where $ c_i<d_i,\,i=1,\dots,r$ such that

$1)$  $\sigma_i$ is admissible for $\sigma_i\circ\dots\circ\sigma_1(g)$, $i=1,\dots, r$ and $\sigma_r\circ\dots\circ\sigma_1(g)=f$

$2)$  $d_1<d_2<\dots<d_r$ 

$3)$ $c_i\ne d_j,\, 1\le i,j\le r $

\end{corollary}

 The above corollaries  can be used to calculate the irreducible subquotients  of Kac modules.

\begin{definition} $
\mathcal K(g)=\{f\mid Hom_{\frak g}(P(f),K(g))\ne0\}
$
\end{definition}
\begin{example}
  Let $m=n=2$ and $g^{-1}(\times)=\{2,3\}$. We are going to describe the set  $\mathcal K(g)$.
  
As the first step we are going to find a transpositions $\pi_{a}^b$ such that $b\in g^{-1}(\times),$ and $\pi_a^b$ is admissible for $\pi_{a}^b(g).$ And it is easy to see  that  there exists only one such transposition $\pi_{1}^2$.

The next step is to find  a transposition $\pi_a^b$ such that $b\in \pi_1^2(g)^{-1}(\times),\, \pi_a^b$  is admissible for $\pi_a^b\circ \pi_1^2(g)$  and $a<1$. It is easy to check that there exists only one such transposition $\pi_0^3$.
So we have  
$$
\mathcal K(g)=\left\{g,\, \pi_1^2(g),\,\pi_{0}^3\circ\pi_1^2(g)\right\}
$$
\end{example}

\begin{remark} We should mention that our algorithm is  essentially the same as in the paper \cite{MS}.
Legal move of weight zero
 $
 g\xrightarrow{[b,a]} f,\,\, a<b
 $ in the sense of \cite{MS} is the same  as $\sigma= \pi_{a}^b$ is an admissible transposition for $f=\sigma(g)$. And a regular increasing pass  from $g$ to $f$ is the same as the sequence  of transpositions 
 $$
 \sigma_1=\pi_{a_1}^{b_1},\dots,\sigma_r=\pi_{a_r}^{b_r},\quad a_1<b_1,\dots, a_r<b_r
 $$ such that
 
$1)$  $\sigma_i$ is admissible for $\sigma_i\circ\dots\circ\sigma_1(g)$, $i=1,\dots, r$ and $\sigma_r\circ\dots\circ\sigma_1(g)=f$

$2)$  $b_1<b_2<\dots<b_r$ 

$3)$  $a_i\ne b_j,1\le i,j\le r$.
 
\end{remark}

\section{A bilinear form on the ring $P_{m,n}$}

 In this section we are going to define a bilinear form on the ring of Laurent polynomials $P_{m,n}$   and connect this bilinear form with the canonical  bilinear form on the Grothendieck  ring of Lie superalgebra  $\frak g=\frak{gl}(m,n)$. Let $p\rightarrow p^*$ be the followinng  automorphism  of $P_{m,n}$
$$
x_i^*=x_i^{-1},\,i=1,\dots,m,\quad y_j^*=y_j^{-1},\,j=1,\dots,n
$$
\begin{definition} Let us set 
$$
\Delta(x)=\prod_{i>j}\left(1-\frac{x_i}{x_j}\right),\, \Delta(y)=\prod_{i>j}\left(1-\frac{y_i}{y_j}\right),\,\,\Delta(x,y)=\prod_{i,j}\left(1+\frac{y_j}{x_i}\right)
$$
and for $p,q\in P_{m,n}$ let us  define
\begin{equation}\label{form1}
(p,q)=\frac{1}{m!}\frac{1}{n!}\left[ p^*q\frac{\Delta(x)\Delta(x)^*\Delta(y)\Delta(y)^*}{\Delta(x,y)\Delta(x,y)^*}\right]_0
\end{equation}
where $[\,,\,]_0$ means the constant term and  $(\Delta(x,y)\Delta(x,y)^*)^{-1}$ should be understood as
$$
(\Delta(x,y)\Delta(x,y)^*)^{-1}=\frac{(y_1\dots y_n)^m}{(x_1\dots x_m)^n}\left[\prod_{i.j}\left(1+\frac{y_j}{x_i}\right)\right]^{-2}
.$$

\end{definition}

\begin{thm} \label{form} The following equality hold true
$$
\dim Hom_{\frak{g}}(P,L)=(ch P,ch L)
$$
where $P$ is a finite dimensional projective module, $L$ is any finite dimensional module.
\end{thm}
\begin{proof} We are going to prove the Theorem in  several steps.
First we are going to prove that  characters of Kac modules are pairwise orthogonal  with respect to the pairing $(\,,\,)$.
Let $K(f),K(g)$ be two Kac modules and $\chi=(\lambda,\mu) $ and $\tilde\chi=(\nu,\tau)$  are  the corresponding highest weights, where $\lambda,\nu$ are highest weights of $\frak{gl}(m)$ and $\mu,\tau$ are highest weights of $\frak{gl}(n)$. Then we have 
$$
ch K(f)=\Delta(x,y)s_{\lambda}(x)s_{\mu}(y),\,\,
 chK(g)=\Delta(x,y)s_{\nu}(x)s_{\tau}(y),\,\,
$$
where $s_{\lambda}, s_{\mu},s_{\nu}, s_{\tau}$ are Schur functions. Therefore we have  
$$
(f,g)=\frac{1}{m!}\frac{1}{n!}\left[s^*_{\lambda}s^*_{\mu}\Delta(x)^*\Delta(y)^*s_{\nu}s_{\tau}\Delta(x)\Delta(y)\right]_0=\delta_{\lambda,\nu}\delta_{\mu,\tau}
$$
according to the  orthogonality  of Schur polynomials.

Now let $P(f)$ be  the projective cover  of the  irreducible module  $L(f)$ and $K(g)$ be a Kac module. Then we are going to prove that
\begin{equation}\label{equa1}
\dim Hom_{\frak g}(P,K)=(ch P, ch K)
\end{equation}
 We  can suppose that $P^+(f)=P(f) ,\,K^+(g)=K(g)$. In other words the parity of every weight vector coincides with the parity of the weight. We have  
$$
\dim Hom_{\frak{g}}(P(f),K(g))=n_{g,f}
$$
 where $n_{g,f}$ is the multiplicity of irreducible module $L(f)$ in the  Jordan - Helder series of the module  $K(g)$.  On the other hand from the orthonormality    of Kac modules  it follows that $(P(f),K(g))=m_{f,g}$, where $m_{f,g}$ is the multiplicity of Kac module  $K(g)$ in the Kac flag of the module $P(f)$.
But by BGG reciprocity  $m_{f,g}=n_{g,f}$  and we proved  equality (\ref{equa1}).

To complete the proof, it just remains to show that the following equality 
$$
\dim Hom_{\frak g} (P(f), L)=(ch\,P(f),ch\, L)
$$
is true for any finite dimensional module $L$.
For this, we give two different arguments, the first based on  a fact proved by Serganova in \cite{Serga1} and the second using instead completion in the spirit of Brundan (\cite{Brun1}  \S, 4c).

Now let  $L$ be a module which has a Kac flag. Then 
$$
\dim Hom_{\frak g} (P(f), L)=\sum_{g}\dim Hom_{\frak g} (P(f),K(g))=
$$
$$
\sum_{g}(ch\, P(f),ch\,K(g))=(ch\,P(f),ch\, L)
$$
where $K(g)$ runs over all subquotients of  $L$

Now let $L$ be any  finite dimensional module and $P$ be a projective module.  By Serganova \cite {Serga1} there exist a resolvent of $L$ 
\begin{equation}\label{res}
\dots \rightarrow K_i\rightarrow K_{i-1}\rightarrow \dots \rightarrow K_1\rightarrow L\rightarrow 0
\end{equation}
where  every $K_i$ has a flag of Kac modules. Therefore  we have  an exact sequence of  vector spaces
$$
\dots \rightarrow Hom_{\frak g}(P,K_i)\rightarrow  \dots \rightarrow Hom_{\frak g}(P, K_1)\rightarrow Hom_{\frak g}(P, L)\rightarrow 0
$$

 For any finite dimensional module $V$   let us denote by $wt(V)$ the set of  the weights of  the module  $V$.  Let $N$ be such that  for any $i> N$ we
have $wt(P)\cap wt(K_i)=\emptyset.$ Then for any $i> N$ we have $Hom_{\frak g}(P,K_i)=0$ 
and 
\begin{equation}\label{e1}
dim(P,L)=dim(P,K_1)-dim(P,K_2)+\dots+(-1)^{i+1}dim(P,K_i)
\end{equation} 
 On the other hand from  equality (\ref{res}) we have 

$$
sch L-sch K_1+sch K_2-\dots+(-1)^isch K_i+\dots=0
$$
The above sum makes sense since every weight entries the sum with finite multiplicity.  

Now let us calculate $(ch P, ch K_i)$. We have by definition
$$
(ch P, ch K_i)=\frac{1}{m!}\frac{1}{n!}\left[(ch P)^*ch K_i\frac{\Delta^*(x)\Delta(y)^*\Delta(x)\Delta(y)}{\Delta(x,y)^*\Delta(x,y)}\right]_0=
$$
$$
=\frac{1}{m!}\frac{1}{n!}\left[(ch P)^*ch K_i\Delta^*(x)\Delta(y)^*\Delta(x)\Delta(y)\sum\prod\left(\frac{y_j}{x_i}\right)^{n_{ij}}\right]_0
$$
Now let us take $M$ such that for any $i>M$  all  monomials of the polynomial  $(ch P)^*ch K_i\Delta^*(x)\Delta(y)^*\Delta(x)\Delta(y)$  were negative degree with respect to $x_1,\dots,x_m$. Therefore all  monomials in the above expansion  have negative degree with respect  to $x_1,\dots,x_m$. Therefore 
$(ch P, ch K_i)=0$. Therefore  for $i>M$  we have 
\begin{equation}\label{e2}
(ch P,ch L)-(ch P,ch K_1)+\dots+(-1)^i(chP, ch K_i)=0
\end{equation}
Therefore  if we take $i>\max\{N,M\}$ then from the equalities (\ref{e1}), (\ref{e2}) we have $\dim Hom_{\frak g}(P,L)=(ch P, ch L)$ and Theorem \ref{form} is proved.

Now let us use a completion. For $\chi=\lambda_1\varepsilon_1+\dots+\lambda_m\varepsilon_m+\mu_1\delta_1+\dots+\mu_n\delta_n$ let us set $m(\chi)=\mu_1+\dots+\mu_n$. Let  $K(\mathcal F)_d$ be  the subgroup of  the Grothendieck group $K(\mathcal F)$  generated by $\{L(\chi)\}$ for $\chi\in P^+$ with $m(\chi)\ge d$. We know that $[K(\chi)]$  is the finite linear combination of $L(\chi)$ and  $[L(\tilde\chi)]$ where $\tilde\chi<\chi$. Therefore we can find the sequence  $\{A_i\}_{i\ge1}$ of finite subsets in $P^+$  such that $A_i\subset A_{i+1}$  and 
$$
[L(\chi)]-\sum_{\tilde\chi\in A_i}c_{\tilde\chi}[K(\tilde\chi)]\in K(\mathcal F)_{d_i}
$$
where  $d_1<d_2<d_3\dots.$ It is easy to see that for given projective module $P$ there exists $N_1$ such that   for any $d\ge N_1$ we have $\dim_{\frak g}(P, L)=0$  for any irreducible module $L\in K_d$. And it follows from  formula (\ref{form1}) that  there exists $N_2$ such than for any $d\ge N_2$ we have  $(ch P, ch L)=0$ for  any irreducible module $L\in K_d$. Therefore for $d_i\ge\max\{N_1,N_2\}$ we have 
$$
\dim Hom (P, L(\chi))=\sum_{\tilde\chi\in A_i}c_{\tilde\chi}\dim Hom(P,K(\tilde\chi))=
$$
$$
=\sum_{\tilde\chi\in A_i}c_{\tilde\chi}(P,K(\tilde\chi))=(P,L(\chi))
$$
and we proved the Theorem in this way.
\end{proof}
\begin{corollary} 
$$
\mathcal P(f)=\{g\mid (ch\,P(f),ch\,K(g))\ne0\}
$$
\end{corollary}

\section{Kac modules and Euler  characters}

Now we are going to calculate the number $(ch\,K(f),ch\,E(g))$ where $K(f)$ is a Kac module and $E(g)$ is an Euler virtual  module.  General formula for Euler characters was given by V. Serganova in \cite {Serga1}. For any parabolic subalgebra  $\frak{p}\subset\frak{g}$ and any finite dimensional module $M$ of $\frak{p}$ by a super version of Borel - Weil - Bott construction one can define the  virtual Euler module  $ E^{\frak p}(M)$. According to the general formula due to Serganova \cite {Serga1}
$$
ch E^{\frak p}(M)=\sum_{w\in W}w\left(\frac {De^{\rho} ch M}{\prod_{ \alpha\in R_{\frak{p}}\cap R_1^+}(1-e^{\alpha})}\right)
$$
with 
$$
D=\frac{\prod_{\alpha\in R_1^+}(e^{\alpha/2}-e^{-\alpha/2})}{\prod_{\alpha\in R_0^+}(e^{\alpha/2}-e^{-\alpha/2})}
$$
Here $\rho$ is the half-sum of the even positive roots minus the half-sum odd positive roots, $R_{\frak p}$ is the set of roots $\alpha$ such that $\frak g_{\pm\alpha}\subset\frak p$. Consider now  $\frak g=\frak{gl}(m,n)$ and let $(r,s)$ be  a pair of integers  such that $0\le r \le m,\,$ $ 0\le s\le n,\,\,r-s=m-n$. We will denote the set  of such pairs as  $P(m,n)$. Next let us choose  for $(r,s)\in P(m,n)$ the following system of simple roots
$$
\{\varepsilon_i-\varepsilon_{i+1},\delta_j-\delta_{j+1}, \varepsilon_r-\delta_1,\delta_s-\varepsilon_{r+1},
\varepsilon_m-\delta_{s+1}\},\,i\in[1,m]\setminus\{r\},\,j\in[1,n]\setminus\{s\}.
$$
So we have the corresponding set of positive even and odd roots. 
 Consider  now the parabolic subalgebra  $\frak p$ with
$$
R_{\frak p}=\{\varepsilon_i-\varepsilon_j,\,\delta_p-\delta_q,\, \pm (\varepsilon_i-\delta_p)\},
$$
where  $r+1\le i,j\le m,\,i\ne j$ and $s+1\le p,q\le n,\,p\ne q$.

If we set 
$$
\chi_{r,s}=\sum_{i=1}^r\tau_i\varepsilon_i+\sum_{j=1}^s\nu_j\delta_j
$$
where 
$$
\tau=(\tau_1,\dots,\tau_r),\,\,\nu=(\nu_1,\dots,\nu_s)
$$
are non increasing sequences of integers
then $\chi$ defines one dimensional representation of $\frak p$.
For any function $f(x_1,\dots,x_m,y_1,\dots, y_n)$ let us define the following alternation operation
$$
\{f(x,y)\}=\sum_{w\in S_m\times S_n}\varepsilon(w)w(f(x,y)).
$$
Then  it is easy to check that  Euler character is given by the following  formula
 $$
ch\,E(\chi_{r,s})\Delta(x)x^{\rho_m}\Delta(y)y^{\rho_n}
$$
\begin{equation}\label{Eulerch}
=\left\{\prod_{(ij)\in D_{+}}\left(1+\frac{y_j}{x_i}\right)\prod_{(ij)\in D^{-}}\left(1+\frac{x_i}{y_j}\right)x^{\tau}y^{\nu}x^{\rho_m}y^{\rho_n}\right\}
\end{equation}
where 
$$
D_{+}=[1,r]\times[1,n],\quad D_{-}=[ r+1,m]\times[1,s].
$$
\begin{remark} If we apply to the formula (\ref{Eulerch})  the  automorphism  $\omega$ which acts identically on $x_1,\dots,x_n$ and acts  multiplication by $-1$ on $y_1,\dots,y_m$ then we get  the Euler supercharacter (see \cite{Ser1}  Proposition 5.10). And it was proved in \cite{Ser1} that Euler supercharacters  $\omega(E(\chi_{r,s}))$
where $(r,s)\in P(m,n)$ form a basis in the ring of superchsracters. Therefore $ch E(\chi_{r,s})$ where $(r,s)\in P(m,n)$ form a basis in the ring $K(\mathcal F)$.
\end{remark}

As before we  can use diagram $g=(A,B)$ where

$$
 A=\{\tau_1,\tau_2-1,\dots,\tau_r+1-r\},\,
$$
$$
 B=\{s-r-\nu_s,s-r-\nu_{s-1}-1,\dots,-r-\nu_1\}
$$

 As a particular case we have the formula for character  of Kac module $K(\tilde\chi)$ where
   $\tilde\chi=(\lambda,\mu)$ and
$
\lambda=(\lambda_1,\dots,\lambda_m),\,\,\mu=(\mu_1,\dots,\mu_n),
$
are non increasing sequences of integers. 
 In this case  we have 
 $$
 ch K(\tilde\chi)=\Delta(x,y)s_{\lambda}(x)s_{\mu}(y)
 $$
 and the corresponding diagram $f=(\tilde A,\tilde B)$ where

$$
\tilde A=\{\lambda_1,\lambda_2-1,\dots,\lambda_m+1-m\},\,
$$
$$
\,\tilde B=\{n-m-\mu_n,n-m-\mu_{n-1}-1,\dots,-m-\mu_1\}
$$

\begin{definition} Let $X,Y$ be two sets of integers  such that $X\cap Y=\emptyset$. Let $ x_1>x_2>\dots,x_m$ be  the elements of $X$ in decreasing order  and $ y_1>y_2>\dots>y_n$ be  the elements of $Y$ in decreasing  and  $ z_1>x_2>\dots,z_{m+n}$ be the elements of $Z=X\cup Y$ in decreasing order. The sign of a permutation $\sigma$ such that 
$$
\sigma(x_1,\dots,x_m,y_1,\dots,y_n)=(z_1,\dots,z_{n+m})
$$
will be denoted by $\varepsilon(X,Y)$.

\end{definition}
 Let us keep the notation of the above definition. Then the following Lemma can be easily proved.
\begin{lemma}\label{sign} Let us set 
$$
a_i=|X\cap(-\infty, x_i)|,\, \,i=1,\dots,m \quad b_j=[Y\cap(y_j,+\infty)|,j=1,\dots,n.
$$
where $|A|$ means the cardinality of $A$.
Then the following equalities hold true
$$
\varepsilon(X,Y)=(-1)^{a_1+\dots+a_m}=(-1)^{b_1+\dots+b_n}
$$
\end{lemma}
\begin{definition} Let $h$ be a diagram and $C\subset h^{-1}(\circ)$. Then by $h*C$ we will denote  the following diagram
$$
(h*C)^{-1}(x)=\begin{cases}
h^{-1}(x),\,x=<,>\\
h^{-1}(x)\cup C,\, x=\times\\
h^{-1}(x)\setminus C,\, x=\circ
\end{cases}
$$
\end{definition}
Now we can formulate the main result of this section.
 
\begin{thm} \label{KacEuler} The following statement holds true:
\begin{equation}\label{main1}
(ch\,K(f),ch\, E(g))=\begin{cases}\varepsilon(f,g),\, f=g*C, C\subset g^{-1}({\circ})\cap\Bbb Z_{\le n-m}\\
0,\,\,\text{otherwise}
\end{cases}\end{equation}
where 
$$
\varepsilon(f,g)=(-1)^{\frac12r(r-1)+\frac12m(m-1)+s(m-r)+S(C)}\varepsilon(A,C)\varepsilon(C,B)
$$
 and  $S(C)$ is equal to the sum of the elements of $C$. 
 \end{thm}
\begin{proof}  We have  from the definition of Kac module that 
$$
\frac{ch\,K(f)}{\Delta(x,y)}=s_{\lambda}(x)s_{\mu}(y)
$$
and from the definition of Euler character
$$
\frac{ch\, E(g)\Delta(x)x^{\rho_m}\Delta(y)y^{\rho_n}}{\Delta(x,y)}=\left\{\frac{\prod_{(ij)\in D_{+}}\left(1+\frac{y_j}{x_i}\right)\prod_{(ij)\in D_{-}}\left(1+\frac{x_i}{y_j}\right)}{\Delta(x,y)}x^{\tau}y^{\nu}x^{\rho_m}y^{\rho_n}\right\}
$$
$$
=\frac{(x_{r+1}\dots x_m)^s}{(y_1\dots y_s)^{m-r}}\left\{\prod_{(i,j)\in D_{r,s}}\left(1+\frac{y_j}{x_i}\right)^{-1}x^{\tau}y^{\nu}x^{\rho_m}y^{\rho_n}\right\}
$$
where $D_{r,s}=[r+1,m]\times[s+1,n]$. Therefore
$$
\frac{ch\, E(g)\Delta(x)x^{\rho_m}\Delta(y)y^{\rho_n}}{\Delta(x,y)}=
$$
$$
\sum_{a_1\ge a_2\ge\dots\ge a_{m-r}\ge0}(-1)^{|a|}\left\{\frac{(x_{r+1}\dots x_m)^s}{(y_1\dots y_s)^{m-r}}s_{a}( y_{s+1},\dots,y_n)s_a( x^{-1}_{r+1},\dots,x^{-1}_{m})\right\}
$$
where $|a|=a_1+\dots+a_{m-r}$. Further we have
$$
\left\{x^s_{r+1}\dots x^s_ms_{a}(x^{-1}_{r+1},\dots,x_m^{-1})x_1^{\tau_1}\dots x_r^{\tau_r} x^{\rho_m}\right\}=
$$
$$
\{x_1^{\tau_1}\dots x_r^{\tau_r}x_{r+1}^{s-a_{m-r}}\dots x_m^{s-a_1}x^{\rho_m}\}=
s_{\tau_1,\dots,\tau_r,s-a_{m-r},\dots,s-a_1}\Delta(x)x^{\rho_m}.
$$
In the same way it is easy to see that
$$
\left\{y_1^{r-m}\dots y_s^{r-m}s_{a}(y_{s+1},\dots,y_n)y_1^{\nu_1}\dots y_s^{\nu_s} y^{\rho_n}\right\}=
$$
$$
=s_{\nu_1+s-n,\dots,\nu_s+s-n,\,a_{1},\dots,a_{n-s}}\Delta(y)y^{\rho_n}
$$
Therefore 
$$
(ch\,K(f),ch\,E(g))=
$$
$$
\sum_{a}(-1)^{|a|}(s_{\lambda},s_{\tau,s-a_{m-r},\dots, s-a_1})(s_{\mu},s_{\nu_1+r-m,\dots,\nu_{s}+r-m,a})
$$
 It is easy to check that for given $\lambda$  there exists a unique sequence $a$ and  a permutation $\sigma\in S_{m}$ such that
$$
(\lambda_1,\dots,\lambda_m)+\rho_m=\sigma((\tau_1,\dots,\tau_r,s-a_{m-r},\dots,s-a_1)+\rho_m)
$$
or in an equivalent form $\tilde A=\sigma(A,C)$
where 
$$
C=\{s-r-a_{m-r},\dots, s+1-m-a_1\}.
$$
In the same way there exists a permutation $\tau\in S_n$ such that
$$
(\mu_1,\dots,\mu_n)+\rho_n=\tau((\nu_1+r-m,\dots,\nu_s+r-m,\,a_{1},\dots,a_{n-s})+\rho_n)
$$
or in the equivalent form $\tilde B=w_n\circ\tau\circ w_n(CB)$, where $w_n(i)=n-i+1,\, i=1,\dots, n$. Therefore
$$
(K(f),E(g))=(-1)^{|a|}sign(\sigma)sign(\tau).
$$
But 
$$
S(C)=s(m-r)+\frac12r(r-1)-\frac12m(m-1)-|a|
$$
and the Theorem is proved.
\end{proof}

\begin{corollary}\label{prev} Let $f,g$ be such diagrams that
$
(K(f),E(g))\ne0
$
and $g=(A,B)$.
Let us also suppose that for transposition $\tau=\pi_a^b$  we have $\tau(g)=g,\,a\in f^{-1}(\times)$ and $a,b\le n-m$. Then
$$
(ch\,K(\tau(f)),ch\,E(g))=(-1)^{n_{ab}+m_{ab}+a-b}(ch\,K(f),ch\,E(g))
$$
where  $n_{a,b}=|A\cap (a,b)|,\,m_{a,b}=|B\cap(a,b)|$.
\end{corollary}
\begin{proof} By Theorem \ref{KacEuler} $f=g*C$  where $C\subset g^{-1}({\circ})\cap\Bbb Z_{\le n-m}$. Since $\tau(g)=g$  we have  $\tau(f)=g*\tau(C)$ and by Theorem \ref{KacEuler} we have
$$
(K(f),E(g))=(-1)^{\frac12r(r-1)+\frac12m(m-1)+s(m-r)+S(C)}\varepsilon(A,C)\varepsilon(C,B)
$$
$$
(K(\tau(f)),E(g))=(-1)^{\frac12r(r-1)+\frac12m(m-1)+s(m-r)+S(\tau(C))}\varepsilon(A,\tau(C))\varepsilon(\tau(C),B).
$$
Further we have  the following equalities in  $\Bbb Z_2$:  $S(C)-S(\tau(C))=a-b$ and by Lemma \ref{sign}
$$
 \varepsilon(A,C)-\varepsilon(A,\tau(C))=|A\cap(a,b)|,\,\, \varepsilon(C,B)-\varepsilon(\tau(C),B)=|B\cap(a,b)|.
$$
Lemma is proved.
\end{proof}
\begin{corollary}\label{prod} Let $f,g,h$ be  such diagrams that  
$$
(ch\,P(f),ch\,K(g))\ne0,\quad(ch\,K(g),ch\,E(h))\ne0 
$$
and $\tau=\pi_a^b,\,a,b\le n-m$ be an admissible transposition for $f$ such that  $a,b\notin h^{-1}(\times)$. Then
$$
(ch\,K(\tau(g)),ch\,E(h))+(ch\,K(g),ch\,E(h))=0
$$
\end{corollary}
\begin{proof} By  Corollary \ref{prev} it is enough to prove that $n_{ab}+m_{ab}+a-b$ is an odd number. Let us denote by $(a,b)_{x}=g^{-1}(x)\cap (a,b)$. Then we have 
$$
(a,b)=(a,b)_{>}\cup (a,b)_{<}\cup (a,b)_{\times}\cup (a,b)_{\circ}
$$
Therefore 
$$
b-a-1=|(a,b)_{>}|+| (a,b)_{<}|+| (a,b)_{\times}|+| (a,b)_{\circ}|.
$$
where $|A|$ means the cardinality  of the set $A$.
But
 $$
 n_{a,b}=|(a,b)_{>}|+| (a,b)_{\times}|,\,  m_{a,b}=|(a,b)_{<}|+| (a,b)_{\times}|,\,
$$
Therefore it is enough to prove that 
$
| (a,b)_{\times}|+| (a,b)_{\circ}|  
$
is an even number. Let $C=\{c_1,\dots, c_r\}$.  We have 
$$
\varphi\circ g=\varphi\circ f+2\sum_{i=1}^r(\delta_{d_i}-\delta_{c_i})
$$
and by definition  admissible transposition   we have $\sum_{i\in (a,b)}\varphi\circ f(i)=0$. Therefore
$$
| (a,b)_{\times}|-| (a,b)_{\circ}|=\sum_{i\in (a,b)}\varphi\circ g(i)=2\sum_{i=1}^r(\delta_{d_i}-\delta_{c_i}). 
$$
Corollary is proved.
\end{proof}
\section{Partially polynomial representations}
 
\begin{definition} A weight $\chi\in P$ 
$$
\chi=\lambda_1\varepsilon_1+\dots+\lambda_m\varepsilon_m+\mu_1\delta_1+\dots+\mu_n\delta_n
$$is called partially polynomial  (in $y_1,\dots, y_m$) if $\mu_1,\dots,\mu_m\in \Bbb Z_{\ge0}$.
\end{definition}
\begin{corollary} A diagram  $f=(A,B)$ corresponds to the partially polynomial  highest  weight if and only if all elements of $B$ are not grater than  $n-m$.
\end{corollary}
\begin{definition} A representation $V$ of $\frak{gl}(m,n)$ is called partially polynomial (in $y_1,\dots, y_n$) if all its weights are partially polynomials or  its character is a polynomial  in $y_1,\dots, y_n$.  
\end{definition}

We should note that there is no loss of generality in restricting our attention to partially polynomial representations, since an arbitrary finite dimensional irreducible representation of $\frak{gl}(m,n)$  can be obtained from some  partially polynomal representation  by tensoring with a one dimensional representation.
\begin{example} Standard representation with the character $x_1+\dots+x_n+y_1+\dots+y_m$ is partially polynomial.  One dimensional representation with  character $\frac{y_1\dots y_n}{x_1\dots x_m}$ is also partially polynomial representation.
\end{example}

The subcategory of the modules with partially polynomials weights will be denote by  $\mathcal F^+$.

\begin{definition}  For any $M\in \mathcal F$ let us denote by $M^{-}$ the submodule generated by all weight vectors with non  partially polynomials weights. Let us also define a functor  $F^+: \mathcal F\rightarrow \mathcal F^+$ by the following formula
$$
F^+(M)=M/M^-
$$
\end{definition}

\begin{lemma}\label{exact} 
$1)$ Functor $F^+$ is right exact.

$2)$ Functor $F^+$ maps projective objects in $\mathcal F$ to projective objects in $\mathcal F^+$.
\end{lemma}

\begin{proof}  

$1)$
By definition of $M^{-}$ we have the following equality for all $N\in \mathcal F^+$
$$
Hom_{\mathcal F}(M,N)=Hom_{\mathcal F^+}(F(M),N)
$$
Consider a  functor $ G: \mathcal F^+\rightarrow \mathcal F$   such that $G(N)=N$. Then the above equality means that $G$ is right adjoint to $F$. Therefore $F$ is right exact. 

$2)$ follows from  $1)$.
\end{proof}

\begin{lemma} The following statements hold true

$1)$ Let $\chi\in P^+$ and $ K(\chi)$ be  the corresponding  Kac module, then
$
F^+(K(\chi))=K(\chi)
$ if $\lambda$ is a partially polynomial weight and $0$ otherwise.

$2)$ Let $L(\chi)$ be the irreducible finite dimensional module corresponding to the weight $\chi$. Then $F^+(L(\chi))=L(\chi)$ if $\chi$ is a partially polynomial weight and $0$ otherwise.
\end{lemma}
\begin{proof}  
$1)$ Let $\chi$ be a partially polynomial highest weight. Therefore $\chi-\alpha$ is also partially polynomial for any positive root $\alpha$. Therefore all weights of the module $K(\chi)$ are partially polynomials, so $K(\chi)^-=0$ and $F^+(K(\chi))=K(\chi)$. If $\chi$ is not partially polynomial then $K(\chi)^-=K(\chi)$ since it is generated by the vector of the weight $\chi$. Therefore $F^+(K(\chi))=0$. 

$2)$  Let  $\chi$ be a  partially polynomial weight. Since $L(\chi)$ is a quotient  of $K(\chi)$ then by the first statement we have $F^+(L(\chi))=L(\chi)$.  If $\chi$ is not partially polynomial then by the first statement of the Lemma and by Lemma \ref{exact}  $F^+(L(\lambda))=0$.
\end{proof}
\begin{corollary} Let $M\in \mathcal F$  and suppose that it has composition series of Kac modules and  in  the Grothendieck group of $\mathcal F$   we have 
$$
[M]=\sum_{\lambda\in I}m_{\chi}[K(\chi)]
$$
Then in the Grothendieck group of $\mathcal F^+$ we have 
$$
[F^+([M])=\sum_{\chi\in I^{pol}} m_{\chi}[K(\chi)]
$$
where $I^{pol}$ is a subset of partially polynomial weights of $I$.
\end{corollary}

\section{Projective covers and Euler  characters}
  In this section we are going to give an algorithm  to represent   Euler  characters as the sum  of  characters of irreducible modules. For any finite dimensional module $V$ we have the following formula in the ring $\Lambda^{\pm}_{m,n}$
 $$
 ch\,V=\sum_{P}(ch\, P,\,ch\,V)ch\,L_P
 $$
 where sum is taken over all projective covers and $L_P$ is  the irreducible module corresponding to $P$.
 
 So if $V=E(h)$ is an Euler character then we need to calculate $(ch\, P(f),\,ch\,E(h))$ for fixed $h$ and  all $f$. In order to do so we calculate first $(ch\, P(f),\,ch\,E(h))$ for fixed $f$ and  all $h$. 
 \begin{definition} 
 Let  $f$  be a diagram. Let us set
 $$
f^{-1}_0(\times)=\{a\in f^{-1}(\times)\mid  \pi_a^b\,\,\text {is admissible and}\,\, b\le n-m\},
$$
$$
f^{-1}_1(\times)=\{a\in f^{-1}(\times)\mid  \pi_a^b\,\,\text {is admissible and}\,\, b> n-m\},
$$

\end{definition}
\begin{definition}
Let us   set
$$
 \mathcal E(f)=\{h\mid (ch\,P(f),ch\,E(h))\ne0\}
$$
 \end{definition}
 \begin{definition}\label{prod1} Let $f$ be a diagram and $B\subset f^{-1}(\times)$. We will  denote by  $f_B$ the following diagram
$$
f^{-1}_B(x)=\begin{cases} f^{-1}(x),\,x=<,>\\
f^{-1}(\times)\setminus B\\
f^{-1}(\circ)\cup B\\
\end{cases}
$$
And for a diagram $f$ we will  denote by $f_{>d}$ the diagram $f_B$ in the case when  $B=f^{-1}(\times)\cap \Bbb Z_{>d}$. In other words $f_{> d}$ is the diagram which can be obtained from $f$ by deleting from $f^{-1}(\times)$ all  the numbers which are strictly grater that $d$ and adding them to $f^{-1}(\circ)$.
\end{definition} 

 The following Theorem describes the pairing between projective covers and Euler characters.
 
 \begin{thm} \label{Euler1}  The following equalities hold true

$1)$ 
\begin{equation}\label{pq}
 \mathcal E(f)=\{h\mid h*A\in \mathcal P(f),\, A\subset f_1^{-1}(\times)\cap \Bbb Z_{\le n-m}\}
 \end{equation}
 
 $2)$ If $h\in \mathcal E(f)$ then
 $$
 (ch\,P(f), ch\,E(h))=(ch\,K(h*A),ch\,E(h))
 $$
 \end{thm}
\begin{proof}  Let us prove the first statement. We will denote by $\mathcal Q$ the write hand side of the equality (\ref{pq}).   Let $h\in \mathcal E(f)$. We are going to prove that $h\in \mathcal Q$. Let us denote   by $\mathcal A(h)$   the set of all $ g$ such that 
 $$ 
( ch\,P(f),\,ch\,K( g))\ne0,\,\,\,(ch\,K(g),ch\,E(h))\ne0
 $$
 Then we have 
$$
(ch\,P(f),\,ch\,E(h))=\sum_{g\in\mathcal A(h)}(ch\,K(g),ch\,E(h))
$$
Since $h\in \mathcal E(f)$ the set $\mathcal A(h)$ is not empty and  there exists  $g\in \mathcal A(h)$.   By Theorem \ref{KacEuler} we have 
 $$
 g=h*A,A\subset \Bbb Z_{\le n-m}
$$
So we only need to prove that $A\subset f_1^{-1}(\times)$.
If $A=\emptyset$ then $A\subset f_1^{-1}(\times)$.  Let $A\ne\emptyset$ and  $a\in A$.
  There are two cases $a\notin f^{-1}(\times)$ and $a\in f^{-1}(\times)$. 
 
 Let us consider the first case. Let $\tau= \pi_{c}^a$ be the corresponding admissible transposition then $c\notin g^{-1}(\times)$ and therefore  $a\notin h^{-1}(\times)$.
 By Corollary \ref{prod} the set $\mathcal A(h)$ is invariant under the action of $\tau$ and  for any $\tilde g\in \mathcal A(h)$ we have 
 $$
(ch\, K(\tilde g), ch E(h))+(ch\, K(\tau (\tilde g)), ch E(h))=0
$$
Therefore  $(ch\,P(f),\,ch\,E(h))=0$ and the first case is impossible.

Consider the second case $a\in f^{-1}(\times)$. Let $\tau=\pi_a^b$ be the corresponding admissible transposition then $b\notin  g^{-1}(\times)$ and therefore $b\notin h^{-1}(\times)$. We have two possibilities $b\le n-m$ and $b>n-m$. If $b\le n-m$ then in the same way as above we can prove that  $(ch\,P(f),\,ch\,E(h))=0$.  So the only possibility left $b>n-m$.  Therefore $A\subset f^{-1}_1(\times)$ and $\mathcal E(f)\subset\mathcal Q$.

Now let us prove the opposite inclusion. Let $h*A\in \mathcal P(f)$ and $A\subset f^{-1}(\times)\cap \Bbb Z_{\le n-m}$. Suppose   that $h*\tilde A\in \mathcal P(f)$. Clearly $(K(h*A),E(h))\ne 0$. Let  $g\in\mathcal P(f)$ such that $(ch\,K(g), ch\,E(h))\ne0$. Then by Theorem \ref{KacEuler} we have $g=h*\tilde A,\,\tilde A\subset \Bbb Z_{n-m}$. Let $a\in A$ and  $\pi_{a}^b$ be the corresponding admissible transposition. Then one of the elements $a,b$ belongs to $\tilde A$.
Suppose that  $a\notin \tilde A$.  Therefore $b\in \tilde A$  but it is impossible since $b>n-m$.  So $a\in \tilde A$ . Therefore $A=\tilde A$. So we see that
$$
(ch\,P(f), ch\,E(h))=(ch\,K(h*A),ch\,E(h))\ne0
$$
So we proved the inclusion $\mathcal E(f)\supset\mathcal Q$ and the second statement. The Theorem is proved.
 \end{proof}
 
Now we are going to investigate the case of partially polynomial representations in more details. 
\begin{definition}
Let us  denote  by $\mathcal E^{+}(f)$ the set of partially polynomial  diagrams $h$ such that  $(ch\,P(f),ch\,E(h))\ne0$.
\end{definition}
\begin{corollary}\label{cor}  Let $f$ be a partially polynomial diagram then the following equality holds true
$$
\mathcal E^+(f)=\{h\mid h=\pi_C(f)_{>n-m},\,C\subset f^{-1}(\times)\}
$$
\end{corollary}

 \begin{proof} Let us denote the right hand side the above equality by  $\mathcal R$ and by $\mathcal Q^+$. we will denote the set of partially  polynomial diagrams in $\mathcal Q$, where $\mathcal Q$ is the same as in the proof of Theorem \ref{Euler1}. By definition  we have 
 $$
 \mathcal E^+(f)=\mathcal Q^+
 $$
 and we need to prove that $\mathcal R=\mathcal Q^+$. Let $h\in\mathcal Q^+$ then 
by Theorem \ref{Euler1} we have $g=h*A\in \mathcal P(f),\,A\subset f_1^{-1}(\times)\cap \Bbb Z_{\le n-m}$ and since $h,f$ are partially polynomial diagrams  then $g$ is a partially polynomial diagram too. Futher  we have 
$$
h=g_{A}=(\pi_{A}(g))_{> n-m}
$$
Besides since $g\in P(f)$ we have $g=\pi_{B}(f),\, B\subset f_0^{-1}(\times)$. Therefore 
$$
h=(\pi_A\pi_B(f))_{>n-m}=(\tau_C(f))_{> n-m},\,\,C=A\cup B
$$
So $h\in \mathcal R$. Now let us take $h\in \mathcal R$. Then by definition 
$$
h=\pi_C(f)_{>n-m}, C\subset f^{-1}(\times)
$$
Let us set $B=f_0^{-1}(\times)\cap C, A=f_1^{-1}(\times)\cap C$. Then $h*A=\pi_B(f)\in\mathcal P(f).$ Therefore $h\in \mathcal Q^+$ and we proved the Corollary.
%$$
%h=(\pi_C(f)_{>n-m}=(\pi_{\tilde C}(f)_{> n-m}
%$$  then $C=\tilde C$. Let us set $C_0=f_0^{-1}(\times)\cap C,\,C_1=f_1^{-1}(\times)\cap C$ and the same for $\tilde C$. Then we have 
%$$
%\pi_{C_0}(f_0^{-1}(\times))\cup\left( f_1^{-1}(\times)\setminus C_1\right)=\pi_{\tilde C_0}(f_0^{-1}(\times))\cup\left( f_1^{-1}(\times)\setminus \tilde C_1\right)
%$$
%Therefore $C=\tilde C$ and we proved  the Corollary. 
\end{proof}

\begin{corollary}\label{irre}  Let  $h$ be  a partially polynomial diagram and 
$$
ch\, E(h)=\sum_{f}b_{f,h}ch\, L(f)
$$ 
be the decomposition of Euler character $E(h)$ in terms of characters of irreducible modules.
 Then $b_{f,h}=0,\pm1$ and it is nonzero if and only if  there exists the sequence of transpositions 
 $$
 \sigma_1=\pi_{c_1}^{d_1},\dots, \sigma_s=\pi_{c_s}^{d_s},\,\sigma_{s+1}=\pi_{c_{s+1}}^{d_{s+1}},\dots, \sigma_r=\pi_{c_r}^{d_r}
 $$
 such that
$$
f=\sigma_{r}\circ\dots\circ\sigma_{s+1}\left(\left(\sigma_{s}\circ\dots\sigma_1(h)\right)*\{d_{s+1},\dots, d_r\}\right)
$$
and

$1)$  $\sigma_i$ is admissible for $h_i=\sigma_i\circ\dots\circ\sigma_1(h)$, $i=1,\dots, s$ 

$2)$  $c_1>c_2>\dots>c_{s}$, $d_1,\dots,d_{s}\le n-m$ and $c_i\ne d_j ,\, 1\le i,j\le s$

$3)$ $c_{s+1}>\dots>c_{r}$, $d_{s+1},\dots,d_{r}> n-m$ 

$4)$ $\sigma_i$ is admissible for  $h_i=\sigma_i(h_{i-1}*\{d_i\})$, $i= s+1,\dots,r$

$5)$ 
$
\{c_1,\dots,c_s\}\cap\{c_{s+1},\dots,c_r\}=\emptyset,\,\,\, \{c_1,\dots,c_r\}\cap\{d_1,\dots,d_s\}=\emptyset
$

 \end{corollary}
\begin{proof}   Suppose that all conditions of the Corollary are fulfilled.  Then
$$
f=\sigma_{r}\circ\dots\circ\sigma_{s+1}\left(\left(\sigma_{s}\circ\dots\sigma_1(h)\right)*\{d_{s+1},\dots, d_r\}\right)
$$
Since $d_1,\dots,d_{s}\le n-m$ and $d_{s+1},\dots,d_{r}> n-m$ we can rewrite the above equality in the form
$$
f=\sigma_{r}\circ\dots\circ\sigma_{s+1}\circ\sigma_{s}\circ\dots\sigma_1(h*\{d_{s+1},\dots, d_r\})
$$
Now we are going to prove that $\sigma_1,\dots,\sigma_r$ are admissible for $f$. Let us set 
$$
\tau_i=\sigma_{r-i+1}, a_i=c_{r-i+1},\,b_i =d_{r-i+1}\,\,\,1\le i\le r
$$
Then we have 
$$
h*\{b_{r-s},\dots,b_1\}=\left(\tau_r\circ\dots\circ\tau_{r-s+1}\circ\tau_{r-s}\circ\dots\circ\tau_1(f)\right)
$$
 Again from the conditions of the Lemma it follows that $\tau_i,\, i=1,\dots,r$ is admissible for $f_i=\tau_{i-1}\circ\dots\circ\tau_1(f)$
% And we only need to check, that  conditions (\ref{cond}) are fulfiled.  
We have  $a_{r-s+1}<a_{r-s+2}<\dots<a_r$ and by our assumptions $a_{r-s+i}\ne b_{r-s+j},\,1\le1,j\le s$. Therefore by Corollary \ref{adm1} $\tau_{r-s+1},\dots,\tau_r$ are admissible for $f_{r-s+1}$. Further again by our assumptions $a_1<\dots<a_{r-s}$ and since $b_1,\dots,b_{r-s}>n-m$ we have $b_1>\dots>b_{r-s}$  Let us take  $ r-s+1\le j\le r$. Then $b_j<b_{r-s}$. Therefore by Corollary \ref{admis} $\tau_1,\dots,\tau_{r-s+1},\tau_j$ are admissible for $f$.

Now let us suppose that $(P(f),E(h))\ne0$. Therefore by  the proof of Corollary \ref{cor}  
$$
h=(\pi_{A_1}\pi_{A_0}(f))_{>n-m},\, A_1\subset f_1^{-1}(\times),\,A_0\subset f_0^{-1}(\times)
$$
Let $A_0=\{a_1,a_2,\dots, a_{r-s}\}$ where $a_1<a_2<\dots<a_{r-s}$ and  also \\ $A_1=\{a_{r-s+1},\dots, a_{r}\}$ where $a_{r-s+1}<\dots<a_{r}$.  Since admissible transpositions pairwise commute we have 
$$
h*\{b_{r-s},\dots,b_1\}=\tau_r\circ\tau_2\circ\dots\circ\tau_1(f)
$$
where $\tau_i=\pi_{a_i}^{b_i}$. Let us check that  conditions (\ref{cond}) are fulfilled. It is enough to verify  that  inclusion $[a_j,b_j]\subset (a_i,b_i)$ is impossible if $i>j$. Indeed if $i\le r-s$ or $j>r-s$ then  $a_j>a_i$ and we get a contradiction. If $i>r-s,\, j\le r-s$ then  $b_i \le n-m$ and $b_j>n-m$ and this is again a contradiction. Therefore conditions  $(\ref{cond})$  are fulfilled and by Theorem \ref{adm} $\tau_i$ is admissible for $f_i=\tau_{i-1}\circ\circ\dots\circ\tau_1(f)$ and we can set
$$
\sigma_{i}=\tau_{r-i+1},\,c_i=a_{r-i+1},\,d_i=b_{r-i+1}\,\,\,i=1,\dots,r.
$$

\end{proof}

\begin{example}  
 Let  $n=m=2$ and $h^{-1}(\times)=-1,\, h^{-1}(<)=h^{-1}(>)=\emptyset$. We are going to describe the set $\mathcal E^+(h).$

In this case  there are two possibilities for the first step.

$1)$   We are going to find a transpositions $\pi_{a}^b$ such that $b\in h^{-1}(\times),$ and $\pi_a^b$ is admissible for $\pi_{a}^b(h).$ And it is easy to see  that  there exists only one such transposition $\pi_{-2}^{-1}$.

$2)$
We also need to find a transposition $\pi_{a}^b$ such that $\pi_a^b$ is admissible for $\pi_{a}^b(h*\{b\})$ and $b>0.$ It is also easy to check that there exist two such transpositions $\pi_0^1,\pi_{-2}^1$.

In the second step  there  are  also two possibilities. 

$1)$ We are going to find a transpositions $\pi_{a}^b$ such that $b\in (\pi_{-2}^{-1}h)^{-1}(\times)$ and $\pi_a^b$ is admissible for $\pi_{a}^b\circ\pi_{-2}^{-1}(h).$ And it is easy to see  that  there is no  such a  transposition which satisfies conditions $1),2),3)$ of Corollary \ref{irre}.

$2)$ We are going to find a transpositions $\pi_{a}^b$ such that  $\pi_a^b$ is admissible for $\pi_{a}^b\circ\pi_{-2}^{-1}(h*\{b\})$ and $b>0$. And it is easy to see  that  there is only one such transposition $\pi_0^1.$

So we have 
$$
\mathcal E^+(h)=\{\pi_0^1((\pi_{-2}^{-1}(h))*\{1\}),\,\,\pi_0^1(h*{1}),\,\,\pi_{-2}^1(h*{1})\}
$$
And it is easy to see that 
$$
E(h)=-L(\pi_0^1((\pi_{-2}^{-1}(h))*\{1\}))-L(\pi_0^1(h*{1}),)-L(\pi_{-2}^1(h*{1}))
$$

\end{example}

\section{Some special classes of irreducible modules}

In this section  we  will only consider diagrams of the form $f=(A,A)$  where $A\subset \Bbb Z_{\le 0}$ and instead of $E(f)$ we will write $E(A)$  for Euler virtual module   and $L(A),\,P(A)$   for irreducible module and for projective indecomposable module correspondently. Below $|A|$ means the number of elements in the set $A$. Our aim in this section is to give an explicit formula for characters of irreducible modules for the most atypical block of Lie superalgebra $\frak{gl}(2,2)$.

\begin{definition} Let us set
$$
\mathcal P_n^{(m)}=\{A\subset \Bbb Z_{\le 0}, |A|=n\mid \exists\ B\subset\Bbb Z_{\le 0},\, |B|\le m, \,(ch\,P(A),ch\,E(B))\ne0 \}
$$
and  denote by $\omega$ the following shift 
$$
\omega: \Bbb Z\rightarrow \Bbb Z,\,\,\, \omega(x)=x-1
$$
\end{definition}
 In the following Lemma we give an inductive description of the set  $\mathcal P_n^{(m)}$.
\begin{lemma} \label{ind}The following formulae hold true for $m<n$
$$
\mathcal P_{n}^{(m)}=\bigcup_{i=0}^{m}\mathcal P_{n,i}^{(m)},\,\quad\mathcal P_{n,i}^{(m)}=\left\{\{-i, C\}\mid C\in \omega^{i+1}( \mathcal P_{n-1}^{(m-i)})\right\}
$$
\end{lemma}
\begin{proof} It is easy to see that
$$
\mathcal P_{n}^{(m)}=\bigcup_{i=0}^{m}\mathcal A_i,\quad  \mathcal A_i=\{A\in \mathcal P_{n}^{(m)}\mid \max_{a\in A}a=i\}
$$
and we only need to show that $\mathcal A_i=\mathcal P_{n,i}^{(m)}$.

First let us note that $A\in \mathcal P_{n}^{(m)}$ if and only if there exist at least $n-m$ elements from $A$ such that  every  corresponding admissible transposition has one positive element.  If in addition $A$ contains $0$ then transposition $\pi_0^1$ is admissible for $A$. Therefore for  the set $A\setminus \{0\}$ there must be at least $n-m-1$ admissible transpositions with one positive element. This proves the equality $\mathcal A_0=\mathcal P_{n,0}^{(m)}$. The same arguments work for any $0<i\le m$. Lemma is proved.
\end{proof}

 In the case $m=1$ we can give an explicit description of the set $\mathcal P_n^{(m)}$.
\begin{lemma} We have $\mathcal{P}^{(1)}_{n}= S_1\cup S_2$ where 
$$
S_1=\{ \{0,-1,-2,\dots,2-n, a\}\mid a\le 1-n\},\,
$$
$$
 S_2=\{\{0,-1,-2,\dots,-n\}\setminus\{b\}\mid b=0,-1,\dots, 1-n,-n\}
$$
\end{lemma}
\begin{proof} Use of Lemma \ref{ind} and 
induction on $n$.
\end{proof}

The following Lemma gives the value of our bilinear form on some pairs of projective modules and Euler characters.

\begin{lemma}\label{rank1}Let $A_0=\{0,-1,-2,\dots,1-n,\,-n\}$.  Then we have 
$$
\mathcal E^+(A_0\setminus\{-n\})\cap\mathcal{P}^{(1)}_{n}=\{E(\emptyset), E(0),\dots,E(1-n)\},\quad (1,1,-1\dots,(-1)^{n-1})
$$
$$
\mathcal E^+(A_0\setminus\{b\}))\cap\mathcal{P}^{(1)}_{n}=\{E(b),\,E(b-1)\},\quad ((-1)^{n-1},(-1)^{n-1}),\,\, b=0,\dots, 1-n
$$
$$
\mathcal E^+(0,-1,\dots,2-n,b)\cap\mathcal P^{(1)}_n=\{E(b+1),\,E(b)\},\, ((-1)^{n-1},\,(-1)^{n-1}),\,b\le -n
$$
we also  indicate  on the right the corresponding value $(ch\,P(A), ch\,E(B))$.
\end{lemma} 
\begin{proof} It easily follows from Theorem \ref{Euler1}.
\end{proof}
\begin{corollary}\label{graph} The following formulae hold true

$1)$
$$
ch\,L(A_0\setminus\{-n\})=ch\,E(\emptyset),\,\, 
$$

$2)$ if 
 $b\in [1-n,0]$ then
$$
ch\,L(A_0\setminus\{b\})=(-1)^{n-1}[ch\,E(b)-ch\,E(b+1)+\dots +
$$
$$
(-1)^{b}ch\,E(0)+(-1)^{b+1}(1-b)ch\,E(\emptyset)]
$$

$3)$ if $b\le -n$ then 

$$
ch\,L(0,-1,\dots, 2-n,b)=
(-1)^{n-1}[ch\,E(b+1)-\dots+
$$
$$
(-1)^{b+1} ch\,E(0)+(-1)^{b}n\, ch\,E(\emptyset)] 
$$
\end{corollary}

\begin{proof}  Let us prove the first statement. It is enough to check that 
\begin{equation}\label{one}
(P(A_0\setminus\{-n\}),E(\emptyset))=1,
\end{equation}
and
\begin{equation}\label{two}
 (P(B),E(\emptyset))=0,\, \text{for any}\, B,\,|B|=n,\,B\subset\Bbb Z_{\ge0},\,B\ne A.
\end{equation}
Equality (\ref{one}) follows from  Lemma \ref{rank1}. Besides if $B\notin \mathcal P^{(1)}_n$ then equality (\ref{two}) follows from  the definition of the set $\mathcal P^{(1)}_n$ and  if $B\in \mathcal P^{(1)}_n$ then this equality  follows from  Lemma \ref{rank1}. The other two statements can be proved in the same manner.
\end{proof}

 \begin{definition} Let us define a linear operator on  the   characters  of Kac modules by the formula
 $$
 T(ch\,K(A))=ch\,K(\omega(A)),\quad \omega(a)=a-1
 $$
 \end{definition}
 It is easy to see that in the category $\mathcal F^+$ the operator $T$ corresponds to tensor multiplication on one dimensional module with the character $\frac{y_1\dots y_n}{x_1\dots x_m}$.
Now we are going to describe the action of the linear operator $T$   on the Euler characters.
\begin{lemma} The following formulae hold true (we suppose that $A,B\subset\Bbb Z_{\le0})$
\begin{equation}\label{formula2}
T(ch\,E(B))=(-1)^{n+p}\left[ch\,E(\omega(B))-(-1)^pch\,E(\{0\}\cup\omega(B))\right]
\end{equation}
 where $p=|B|$ and we suppose that $E(B)=0$, if the number of elements in $B$ is strictly grater than $n$.
\end{lemma}
\begin{proof} It is enough to prove the following equality for any set $A$ such that $|A|=n$
$$
(ch\,K(A),ch\,T(E(B)))=
$$
$$
(ch\,K(A),\,(-1)^{n+p}\left[ch\,E(\omega(B))-(-1)^pch\,E(\{0\}\cup\omega(B))\right])
$$
So let us calculate separately the left hand side and the right hand side. We have 
$$
(ch\,K(A), ch\,TE(B))=\left(ch\,K(A),\frac{y_1\dots y_n}{x_1\dots x_n}ch\,E(B)\right)=
$$
$$
\left(\frac{x_1\dots x_n}{y_1\dots y_n}ch\,K(A),ch\,E(B)\right)
$$
$$
=(ch\,K(\omega^{-1}(A)),ch\,E(B))
$$
But by Theorem \ref{KacEuler}  we have 
$$
(ch\,K(\omega^{-1}(A)),ch\,E(B))=\begin{cases}0,\,\,\text{if}\,\, B\not\subset\omega^{-1}(A),\,\text{or}\,\,0\in A\\
(-1)^{\frac12p(p-1)+\frac12n(n-1)+S(\omega^{-1}(A)\setminus B)}
\,\text{otherwise}\end{cases}
$$

Let $B\not\subset \omega^{-1}(A)$ then $\omega(B)\not\subset A $ and  $\omega(B)\not\subset \{0\}\cup A $. Therefore 
$$
\left(ch\,K(A),ch\,E(\omega(B)\right)=0,\,\, \left(ch\,K(A),ch\,E(\{0\}\cup\omega(B)\right)=0
$$
 Suppose that $0\in A$ and $B\subset \omega^{-1}(A)$. Then we have $\{0\}\cup\omega(B)\subset A$ and therefore
$$
\left(ch\,K(A),ch\,E(\omega(B)\right)=(-1)^{\frac12p(p-1)+\frac12n(n-1)+S(A\setminus\omega(B))}
$$
$$
 \left(ch\,K(A),ch\,E(\{0\}\cup\omega(B)\right)=(-1)^{\frac12p(p+1)+\frac12n(n-1)+S(A\setminus\omega(B)\cup\{0\})}
$$
So formulae (\ref{formula2}) holds true in this case. 

Now consider the last possible case $0\notin A$ and $B\subset \omega^{-1}(A)$. In this case we have $\{0\}\cup\omega(B)\not\subset A$. Therefore 
$$
 \left(ch\,K(A),ch\,E(\{0\}\cup\omega(B)\right)=0
$$
and
$$
 \left(ch\,K(A),ch\,E(\omega(B)\right)=(-1)^{\frac12p(p-1)+\frac12n(n-1)+S(A\setminus\tau(B))}
$$
Lemma is proved.
\end{proof} 

Now consider the  case of the most atypical  bloc for Lie superalgebra $\frak{gl}(2,2)$ in the category $\mathcal F^+$. In order to give a reasonable description of the irreducible  characters  in this bloc  we need some special type of graphs.

\begin{definition} Let $n\in Z_{\ge0},\,m\in\Bbb Z_{\ge1}$. Let us denote by $\Gamma_{n,m}$ the graph with $n+m$ vertices  that are integers from the segment $[1-n-m,0]$ such that :

$1)$ there exists exactly one edge containing  any two vertices from $[1-n, 0]$;

$2)$ there exists exactly one edge joining every vertex from $[1-n, 0]$ with every vertex from $[-n, 1-n-m]$.

\end{definition}

\begin{definition} For every graph $\Gamma=\Gamma_{n,m}$ let us define the following element of the Grothendieck ring by the formula
$$
\chi(\Gamma)=ch\,E(\emptyset)-\sum_{v}\varepsilon(v)ch\,E(v)-\sum_{e}\varepsilon(e)ch\,E(e)
$$
where $\varepsilon (v)=(-1)^{i}$ if $v=\{i\}$  and $\varepsilon(e)=\varepsilon(v)\varepsilon(u)$ if $e$ contains $v,u$.
\end{definition}
\begin{remark}  It is convenient  to define $\chi(\Gamma_{n,m})$ in the case when $n=-1$.  In such a case we set $\chi(\Gamma_{n,m})=E(\emptyset)$.
\end{remark}
\begin{example} 
$$
\chi(\Gamma_{-1,3})=ch\,E(\emptyset)
$$
$$
\chi(\Gamma_{2,1})=ch\,E(\emptyset)-ch\,E(0)+ch\,E(-1)-ch\,E(-2)+
$$
$$
ch\,E(0,-1)-ch\,E(0,-2)+ch\,E(-1,-2)
$$
$$
\chi(\Gamma_{0,2})=ch\,E(\emptyset)-ch\,E(0)+ch\,E(-1)
$$
\end{example}
\begin{thm} The following equalities hold true
\begin{equation}\label{first}
ch\,L(a,a-1)=\chi(\Gamma_{|a|-1,1}),\,\, a\le0
\end{equation}
\begin{equation}\label{second}
ch\,L(a,b)=(-1)^{a-b-1}\left[\chi(\Gamma_{|a|-1,1})+\chi(\Gamma_{|a|,a-b})\right],\,\, a-b\ge2,\,a\le0
\end{equation}
\end{thm}
\begin{proof}  We are going to use  the functor  $T$.
In the case of $n=2$ the functor acts by the following formulae
$$
T(ch\,E(\emptyset))=ch\,E(\emptyset)-ch\,E(0),\,\, 
$$
$$
T(ch\,E(a))=-ch\,E(a-1)-ch\,E(0,a-1)
$$
and
$$
T(ch\,E(a,b))=ch\,E(a-1,b-1)
$$
It is not difficult to verify the following equality
$$
T(\chi(\Gamma_{n,m}))=\chi(\Gamma_{n+1,m})
$$
Further  we see that $ch\,L(a,a-1)=T^{|a|}(ch\,E(\emptyset))$.  We will prove equality (\ref{first}) induction on $|a|$ . If $a=0$ the equality is trivial $ch\,L(0,-1)=ch\,E(\emptyset)$. Let $|a|>0$ then we have 
$$
T^{|a|}(ch\,E(\emptyset))=T(T^{|a|-1}(ch\,E(\emptyset)))=T(\chi(\Gamma_{|a|-2,1}))=\chi(\Gamma_{|a|-1,1})
$$

Now let us prove  equality  (\ref{second}) also induction on $|a|$. If $a=0$ and $b\le-2$ then by corollary (\ref{graph}) we have 
$$
ch\,L(0,b)=(-1)^{b+1}(2ch\,E(\emptyset)-ch\,E(0)+\dots+
$$
$$
(-1)^{b}ch\,E(b+1))=(-1)^{b+1}[ ch\,E(\emptyset)+\chi(\Gamma_{0,|b|})]
$$ 
If we  apply to both sides of the above formula functor $T^r$  then we get
$$
ch\,L(-r,b-r)=(-1)^{b+1}T^r[ ch\,E(\emptyset)+\chi(\Gamma_{0,|b|})]=
$$
$$
(-1)^{b+1}[\chi(\Gamma_{r-1,1})+\chi(\Gamma_{r,|b|})]
$$
If we replace $r$ by $-a$ and $b$ by $b-a$ we get the statement.
\end{proof}
\section{Acknowledgments}
This work was supported by the Ministry of Science and Higher Education of the Russian Federation in the framework of the basic part of the scientific research state task, project FSRR-2020-0006.

\end{document}